\documentclass[11pt]{article}    
\usepackage{amsmath,amssymb,amsthm,amsfonts}

\title{Saturated fractions of two-factor designs}

\author{Roberto Fontana \\ Department of Mathematical Sciences, Politecnico di Torino
\and Fabio Rapallo \\ Department DISIT, Universit\`a del Piemonte
Orientale \and Maria Piera Rogantin \\ Department of Mathematics,
Universit\`a di Genoa }

\usepackage{fancybox}

\usepackage{multirow}

\usepackage{color}
\usepackage{graphicx}

\date{\today}
\newtheorem{theorem}{Theorem}[section]
\newtheorem{definition}{Definition}[section]
\newtheorem{proposition}{Proposition}[section]
\newtheorem{corollary}{Corollary}[section]
\newtheorem{lemma}{Lemma}[section]
\newtheorem{remark}{Remark}[section]
\newtheorem{example}{Example}[section]
\newcommand{\design}{{\mathcal D}}
\newcommand{\fraction}{{\mathcal F}}

\begin{document}

\maketitle

\begin{abstract}
In this paper we study saturated fractions of a two-factor design under the simple effect model. In particular, we define a criterion to check whether a given fraction is saturated or not, and we compute the number of saturated fractions. All proofs are constructive and can be used as actual methods to build saturated fractions. Moreover, we show how the theory of Markov bases for contingency tables can be applied to two-factor designs for moving between the designs with given margins.

\end{abstract}

\section{Introduction}

The search for estimable designs is one among the major problems in Design of experiments. Given a model, saturated fractions are fractions of a factorial design with as much points as the number of parameters of the model. As a consequence, all information is used to estimate the parameters, leaving no degrees of freedom to estimate the error term. Nevertheless, saturated fractions are of common use in sciences and engineering, and they become particularly useful for highly expensive experiments, or when time limitations force the choice of the minimum possible number of design points. For general references in Design of experiments, the reader can refer to \cite{raktoeetal:81}  and \cite{bailey:08}, where the issue of saturated fractions is also discussed.

In this paper we restrict our attention to two-factor designs under the simple effects model, and we address the problem of studying the fractions with the minimal number of points and characterizing the saturated ones. As we discuss in this paper, this question is relevant for the simple effect model, as a randomly chosen minimal set of design points yields a singular model matrix with very high probability when the number of levels of the factors becomes large.

Our approach is based on two main ingredients. First, we apply tools from Linear Algebra and Combinatorics to characterize the saturated fractions. Some notions, and in particular the definition of $k$-cycle that we will present later, has already been considered in the framework of contingency tables in \cite{kuhntetal:12} for the definition of robust procedures for outliers detection in contingency tables. Second, we identify a factorial design with a contingency table whose entries are the indicator function of the fraction, i.e., they are equal to $1$ for the fraction points and $0$ otherwise. This implies that a fraction can also be considered as a subset of cells of the table. Most of the techniques we present within the latter framework will be based on Algebraic statistics. The application of polynomial algebra to Design of experiments has already been presented in \cite{pistoneetal:01}, but with a different point of view. The techniques used here are mainly based on the notion of Markov bases and were originally developed for contingency tables, both to solve enumeration problems and to make non-asymptotic inference, and to describe the geometric structure of the statistical models for discrete data. A recent account of this theory can be found in \cite{drtonetal:09}.

In this paper, we benefit from the interplay between algebraic techniques for the analysis of contingency tables and some topics of Design of experiments. From this point of view, the identification of a factorial design with a binary contingency table is an essential step. Some recent results in this direction can be found in \cite{aoki|takemura:10}. The connections between experimental designs and contingency tables have also been explored in \cite{fontanaetal:12}, but limited to the investigation of enumerative problems in the special cases of contingency tables from the Sudoku problems. Here, we use a special class of Markov bases, named universal Markov bases, introduced in \cite{rapallo|rogantin:07}, and we show that the structure of these Markov bases is strictly related with the $k$-cycles.

The paper is organized as follows. In Section \ref{sat-cycles} we set some notations, we state of the problem, we define the $k$-cycles in a fraction and we characterize them in terms of orthogonal arrays. In Section \ref{main-res} we prove the main result, showing that the absence of a $k$-cycle is a necessary and sufficient condition for obtaining a saturated fraction, while in Section \ref{num-sat}, we enumerate the saturated fractions showing that their proportion over the whole number of fractions tends to zero as the number of levels increases. Section \ref{mar-bas} is devoted to the computation of the relevant Markov basis for this problem and to showing its connection with $k$-cycles. Finally, in Section \ref{fut-dir} we suggest some future research directions stimulated by the theory presented here.

\section{Saturated designs and cycles} \label{sat-cycles}

\subsection{Notations and basic definitions}

Let $\design$ be a full factorial design  with $2$ factors, $A$ and $B$,  with $I$ and $J$ levels, respectively ($I,J \geq 2$), $\design=[I] \times [J] = \{1,\dots,I\}\times \{1,\dots,J\}$. We consider a linear model on $\design$:
\begin{equation*}
Y_{i,j} = \mu_{i,j} + \varepsilon_{i,j} \ \ \mbox{ for } i \in [I], j\in [J] \, ,
\end{equation*}
where $Y_{i,j}$ are random variables with means $\mu_{i,j}$ and $\varepsilon_{i,j}$ are centered random variables that represent the error terms. In this paper we always consider the simple effect model, where the means $\mu_{i,j}$ of the response variable are written as:
\begin{equation} \label{sem}
\mu_{i,j}= \mu + \alpha_i + \beta_j \ \ \mbox{ for } i \in [I], j\in [J] \, ,
\end{equation}
where $\mu$ is the mean parameter, and
$\alpha_i$ and $\beta_j$ are the main effects of $A$ and $B$, respectively.

We denote by $p$ the number of estimable parameters. Therefore, $p=I+J-1$ in the model of Equation \eqref{sem}. Under a suitable parametrization, the matrix of
this model  is a full-rank matrix with
dimensions $IJ \times (I+J-1)$. In this paper we will use the
following \emph{model matrix}:
\begin{equation}   \label{mat-repr}
X= \left( m_0 \ | \ a_1 \ | \ \ldots \ | \ a_{I-1} \ | \ b_1 \ | \
\ldots \ | \ b_{J-1}  \right) \, ,
\end{equation}
where $m_0$ is a column vector of $1$'s, $a_1, \ldots, a_{I-1}$
are the indicator vectors of the first $(I-1)$ levels of the
factor $A$, and $b_1, \ldots, b_{J-1}$ are the indicator vectors
of the first $(J-1)$ levels of the factor $B$.
It is known that this matrix corresponds to the following reparametrized model:
\begin{multline*}
\mu_{i,j}=\tilde\mu + \tilde\alpha_i + \tilde\beta_j, \ \ \mbox{ for } i \in [I], j\in [J] \, ; \\ \tilde\mu=\mu+\alpha_I+\beta_J, \quad \tilde\alpha_i=\alpha_i-\alpha_I, \quad \tilde\beta_j=\beta_j-\beta_J \, .
\end{multline*}

A subset $\fraction$, or fraction, of a full design $\design$, with minimal cardinality $\#\fraction=p$, that
allows us to estimate the model parameters, is a \emph{main-effect saturated
design}. By definition, the model matrix $X_\fraction$ of a saturated design is
non-singular.

\begin{example} \label{first-ex}
Let us consider the case $I=3$, $J=4$ and the fraction
\begin{equation*}
{\mathcal F} = \left\{ (1,1), (1,2), (2,2), (2,3), (3,3), (3,4) \right\} \, .
\end{equation*}
The model matrix $X$ of the full design and
the model matrix $X_{\fraction}$ of the fraction are given in Figure \ref{ex_fraction}. Is this case, $\det(X_\fraction)=1$.
\begin{figure}
\begin{footnotesize}
\begin{equation*}
X=\bordermatrix{&1 &a_1 &a_2 & b_1& b_2& b_3\cr
                (1,1)&1 & 1 & 0 & 1 & 0 & 0 \cr
                (1,2)&1 & 1 & 0 & 0 & 1 & 0 \cr
                (1,3)&1 & 1 & 0 & 0 & 0 & 1 \cr
                (1,4)&1 & 1 & 0 & 0 & 0 & 0 \cr
                (2,1)&1 & 0 & 1 & 1 & 0 & 0 \cr
                (2,2)&1 & 0 & 1 & 0 & 1 & 0 \cr
                (2,3)&1 & 0 & 1 & 0 & 0 & 1 \cr
                (2,4)&1 & 0 & 1 & 0 & 0 & 0 \cr
                (3,1)&1 & 0 & 0 & 1 & 0 & 0 \cr
                (3,2)&1 & 0 & 0 & 0 & 1 & 0 \cr
                (3,3)&1 & 0 & 0 & 0 & 0 & 1 \cr
                (3,4)&1 & 0 & 0 & 0 & 0 & 0 }
 \quad
X_{\fraction}=\bordermatrix{&1 &a_1 &a_2 & b_1& b_2& b_3\cr
(1,1)&1 & 1 & 0 & 1 & 0 & 0 \cr
(1,2)&1 & 1 & 0 & 0 & 1 & 0 \cr
(2,2)&1 & 0 & 1 & 0 & 1 & 0 \cr
(2,3)&1 & 0 & 1 & 0 & 0 & 1 \cr
(3,3)&1 & 0 & 0 & 0 & 0 & 1 \cr
(3,4)&1 & 0 & 0 & 0 & 0 & 0 }
\end{equation*}
\end{footnotesize}
 \caption{The model matrix $X$ of the full factorial $3 \times 4$ design and the model matrix $X_{\mathcal F}$ of the fraction in Example \ref{first-ex}.} \label{ex_fraction}
\end{figure}
\end{example}

\subsection{$k$-cycles and orthogonal arrays}

As mentioned in the Introduction, in general, the problem of selecting saturated designs is non trivial and our case does not make exception. The key ingredient to characterize a saturated design for two-factor designs is the notion of \emph{cycle}. Here we give a definition in term of Design of experiments.

\begin{definition}\label{def:k-cycle}
A  {$k$-cycle} ($k \geq 2$) is a subset with cardinality $2k$ of a factorial design $I \times J$ with $I,J \ge k$ where each of the $k$ selected levels (among the $I$'s and $J$'s) of each factor has exactly two replications.
\end{definition}

\begin{example}
Some examples of fractions with $k$-cycles are given in Figure \ref{ex-cicli}. As described above, we identify the design points with the cells of a contingency table in order to simplify the presentation.
\begin{figure}
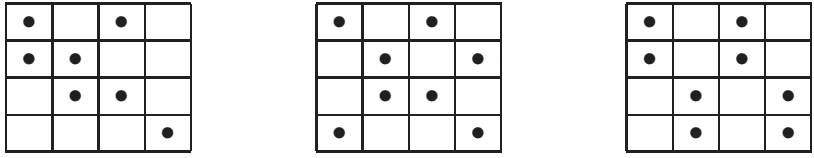

\begin{center}
\begin{tabular}{|c|c|c|c|} \hline
$\bullet$ & & $\bullet$ & \\ \hline
$\bullet$ & $\bullet$ & & \\ \hline
 & $\bullet$ & $\bullet$  & \\ \hline
 & &  & $\bullet$  \\ \hline
\end{tabular} \qquad \qquad
\begin{tabular}{|c|c|c|c|} \hline
$\bullet$ & & $\bullet$ & \\ \hline
 & $\bullet$ & & $\bullet$ \\ \hline
 & $\bullet$ & $\bullet$  & \\ \hline
$\bullet$ & &  & $\bullet$  \\ \hline
\end{tabular} \qquad \qquad
\begin{tabular}{|c|c|c|c|} \hline
$\bullet$ & & $\bullet$ & \\ \hline
$\bullet$ & & $\bullet$ &  \\ \hline
 & $\bullet$ & & $\bullet$  \\ \hline
 & $\bullet$ & & $\bullet$  \\ \hline
\end{tabular}
\end{center}
\caption{A $3$-cycle (left) and two $4$-cycles (middle and right).} \label{ex-cicli}
\end{figure}
The corresponding fractions are:
\begin{eqnarray*}
\fraction_{1}& = & \left\{ (1,1), (1,3), (2,1), (2,2), (3,2), (3,3),(4,3) \right\} \, , \\
\fraction_{2}& = & \left\{ (1,1), (1,3), (2,2), (2,4), (3,2), (3,3), (4,1), (4,4)  \right\} \, ,\\
\fraction_{3}& = & \left\{ (1,1), (1,3), (2,1), (2,3), (3,2), (3,4), (4,2), (4,4) \right\} \, .
\end{eqnarray*}
Notice that $\fraction_{1}$ contains a $3$-cycle, $\fraction_{2}$ and $\fraction_{3}$ contain a $4$-cycle. In $\fraction_3$ the $4$-cycle can be decomposed into two sub-cycles.
\end{example}

In a natural way, the coordinate points of the design $\design$ can be considered as the rows of a $\#\design \times d$ matrix, where $d$ is the number of factors. With a slight abuse of notation, we still call this matrix \emph{design} . The same holds for fractions. In order to analyze the role of the $k$-cycles within the framework of Design of experiments we recall here a combinatorial definition of orthogonal array, see \cite{hedayat|sloane|stufken:1999} and \cite{fcost}.

\begin{definition} \label{def:oa}
A fraction $\fraction$ of a design $I_1 \times \cdots \times I_d$ with $\#\fraction = n$ is an \emph{orthogonal array} of size $n$ and strength $t$ if, for all $t$-tuples of its factors $F_{i_1}, \ldots, F_{i_t}$, all possible combinations of levels in $[I_{i_1}] \times \cdots \times [I_{i_t}]$ appear equally often. We denote such an orthogonal array with $OA( n ; (I_1, \ldots, I_d) ; t)$.
\end{definition}

\begin{proposition} \label{pr:k_OA}
A $k$-cycle ($k \geq 2$) is:
\begin{enumerate}
\item an $OA(2k; (k,k) ;t)$ where $t=2$ if $k=2$, and $t=1$ if $k\geq 3$;
\item the union of two disjoint orthogonal arrays $OA(k; (k,k) ; 1)$.
\end{enumerate}
\end{proposition}
\begin{proof}
\begin{enumerate}
\item This fact follows from Definition \ref{def:k-cycle}. In particular, for $k=2$ $\fraction$ coincides with the full factorial design $2^2$.

\item We construct two disjoint fractions $OA_1$ and $OA_2$,  $OA_1 \cup OA_2 =\fraction$, iteratively. Starting from a given point of the fraction, we assign alternatively to $OA_1$ and $OA_2$ the points of the fraction, choosing the first or the second factor.
\begin{itemize}
\item Choose a point of $\fraction$, say $(i_1,j_1)$, and assign it to $OA_1$.

\item Consider the unique point of $\fraction$ with the same level for the first factor, $(i_1,j_2)$, with $j_2 \ne j_1$, and assign it to $OA_2$.

\item Consider the unique point of $\fraction$ with the same level for the second factor, $(i_2,j_2)$, with $i_2 \ne i_1$, and assign it to $OA_1$.

\item Consider the unique point of $\fraction$ with the same level for the first factor, $(i_2,j_3)$ with $j_3 \ne j_2 \ne j_1$ and assign it to $OA_2$.

\item And so on, until the unique point to choose is already assigned.
\end{itemize}
If not all points of the fraction have been assigned, i.e. if the fraction contains sub-cycles, it is enough to start the procedure above on the remaining points, until all the points are assigned. In this way both $OA_1$ and $OA_2$ have, by construction, exactly one replicate for each of the $k$ levels of the two factors.
\end{enumerate}
\end{proof}

\begin{example} \label{ex:OA}
We show how the decomposition of a fraction into two orthogonal arrays works on a $4$-cycle. Let $I = J = 4$, and consider the $4$-cycle
\begin{equation*}
\fraction_2 = \left\{ (1,1), (1,3), (2,2), (2,4), (3,2), (3,3), (4,1), (4,4)\right\}
\end{equation*}
already considered in Example \ref{ex:OA} and displayed in Figure
\ref{ex-cicli}. The relevant orthogonal arrays are:
\begin{eqnarray*}
OA_{1}& = & \left\{ (1,1), (2,2), (3,3), (4,4)\right\} \, , \\
OA_{2}& = & \left\{ (1,3), (2,4), (3,2), (4,1)\right\} \, .
\end{eqnarray*}
\end{example}

To determine the number of $k$-cycles we need the notion of derangement. A \emph{derangement} is a permutation such that no element appears in its original position.

\begin{proposition}
The number of $k$-cycles is
\begin{equation*}
 \frac{k! \; !k}{2} \, ,
\end{equation*}
where $!k$ denotes the number of derangements of $k$ elements.
\end{proposition}
\begin{proof}
Let us consider $OA_1$ and $OA_2$ as in Proposition \ref{pr:k_OA}. $OA_1$ represents a permutation $\pi_1$ of $[k]$.
The fraction  $OA_2$ represents a derangement of $\pi_1([k])$.
\end{proof}

To actually compute $!k$, recall that $!k$ can be approximated by $\lfloor k!/e+0.5 \rfloor$, where $\lfloor \cdot \rfloor$ is the floor function. For more details on this theory, see for instance \cite{hassani:03}.

\section{$k$-cycles and saturated fractions} \label{main-res}

As mentioned in the previous sections, the connections between saturated designs and cycles have been explored in a slightly different framework, in the study of robust estimators in contingency tables analysis, see \cite{kuhntetal:12}. Nevertheless, we
restate here the relevant theorem within the language of Design of experiments and we give the proof, as its main algorithm will be used later in the paper.

\begin{theorem} \label{th:teo}
A fraction $\fraction$ with $p=I+J-1$ points yields a \emph{saturated
model matrix} if and only if it does not contain cycles.
\end{theorem}

\begin{proof}
In view of Proposition \ref{pr:k_OA}, a cycle can be decomposed into two disjoint orthogonal arrays, $OA_1$ and $OA_2$, of $k$ points each.
\begin{description}
\item [$\Rightarrow$]
Suppose that $\fraction$ contains a cycle. When we sum the rows of the model matrix $X_{\fraction}$ with coefficient $+1$ for the points in the $OA_1$ and with coefficient $-1$ for the points in $OA_2$, we produce a null linear combination and therefore the determinant of $X_{\fraction}$ is zero.

\item [$\Leftarrow$]
Suppose that $X_{\fraction}$ is singular, i.e., there exists a null linear combination of its rows with coefficients that are not all zero. Denote by $r_{(i,j)}$ the row of the model matrix corresponding to the point
$(i,j)$ of the fraction. Therefore, we have
\begin{equation} \label{lincomb}
\gamma_1 r_{(i_1,j_1)} + \cdots + \gamma_p r_{(i_p,j_p)} = 0
\end{equation}
and the coefficients $\gamma_1, \ldots, \gamma_p$ are non all zero. Without loss of generality, suppose that $\gamma_1 >0$. As the indicator vector of the $i_1$-th level of the first factor is in the column span of $X_\fraction$, and the same holds for the $j_1$-th level of the second factor, we must have: $(a)$ a point with the same level for the first factor, say $(i_1,j_2)$, with negative coefficient in Equation \eqref{lincomb}; $(b)$ a point with the same level for the second factor, say $(i_3,j_1)$, with negative coefficient in Equation \eqref{lincomb}. Therefore, there must be a point with level $i_3$ for the first factor and a point with level $j_2$ for the second factor with positive coefficients. Now, two cases can happen: If the point $(i_3,j_2)$ is a chosen point and its coefficient in Equation \eqref{lincomb} is positive, we have a $2$-cycle; otherwise, we iterate the same argument, and we yield a $k$-cycle, with $k>2$.
\end{description}
\end{proof}

\section{The number of saturated designs}  \label{num-sat}

In this section we study the structure of the saturated fractions described in Section \ref{sat-cycles}.

\begin{definition}
Given a fraction $\fraction$, we define its margins $m_A = (m_{A,1}, \ldots ,m_{A,I})$ and $m_B=(m_{B,1}, \ldots ,m_{B,J})$ where:
\begin{eqnarray*}
m_{A,i}= \sum_{(d_1,d_2) \in \fraction} (d_1=i) \quad \text{ for } \; i \in [I] \, , \\
m_{B,j}= \sum_{(d_1,d_2) \in \fraction} (d_2=j) \quad \text{ for } \; j \in [J] \, ,
\end{eqnarray*}
where $( \cdot )$ denotes the indicator function.
\end{definition}

Notice that $m_{A,i}$ is the number of the occurrence in $\fraction$ of the $i$-th level of the factor $A$. For example, the following saturated design
\begin{equation*}
{\mathcal F} = \left\{ (1,1), (1,2), (2,1), (2,4), (3,2), (3,3), (4,4) \right\} \, .
\end{equation*}
has margins $m_A = (2,2,2,1)$ and $m_B=(2,2,1,2)$.

The following lemmas for $I \times I$ designs will be used later in the proof of the main result of this section.

\begin{lemma}\label{le:margin1}
Let $\fraction_I$ be a saturated $I \times I$ design. Its margins satisfy the following conditions:
\begin{enumerate}
\item \label{it:1} $m_{A,+}=m_{B,+}=2I-1$ where $+$ denotes the summation over the corresponding index;
\item \label{it:2} $m_{A,i} \geq 1$, $m_{B,j} \geq 1$ for all $i, j \in [I]$;
\item there exist $i, j \in [I]$ such that $m_{A,i} = m_{B,j} = 1$;
\item \label{it:4} let $i_\star \in [I]$ be an index such that $m_{A,i_\star}=1$. Let $(i_\star,j_\star)$ be the only point $(d_1,d_2)$ of $\fraction_I$ such that $d_1=i_\star$. Then $m_{B,j_\star}>1$.
\end{enumerate}
\end{lemma}
\begin{proof}
\begin{enumerate}
\item It is immediate to see that $m_{A,+}=\sum_{i=1}^I \sum_{(d_1,d_2) \in \fraction} (d_1=i) = \#\fraction_I = 2I - 1$ and the same holds for $B$.

\item Refer to the matrix representation in Equation \ref{mat-repr}. By absurd, suppose that there exists an index $i$ such that $m_{A,i}=0$. We distinguish two cases: if $i \leq I-1$ then $X_{\fraction_I}$ has a null column, corresponding to $a_i$; if $i=I$, the sum of the columns $a_1,\ldots, a_{I-1}$ is equal to $m_0$. In both cases, $X_{\fraction_I}$ is singular. The same applies to $B$.

\item This point follow immediately from items \ref{it:1} and \ref{it:2}.

\item For sake of readability, we suppose that $i_\star=I$ and $j_\star=I$. Thus, the sum of $a_1, \ldots, a_{I-1}$ is equal to the sum of $b_1, \ldots, b_{I-1}$. Hence, $X_{\fraction_I}$ is singular.
\end{enumerate}
\end{proof}

\begin{remark}
The four conditions in Lemma \ref{le:margin1} are not sufficient for characterizing the saturated fractions. A simple counterexample is the following fraction of a $5 \times 5$ design:
\begin{equation*}
\{(1,1), (1,2), (2,1), (2,3), (3,2), (3,3), (4,4), (4,5), (5,4)   \} \, ,
\end{equation*}
that is not saturated, as it contains a $3$-cycle.
\end{remark}

In the following result we analyze how the saturation property is preserved when we add one level to each of the two factors, moving from an $I \times I$ design to an $(I+1) \times (I+1)$ design.

\begin{lemma}\label{le:margin2}
Let $\fraction_I$ be a saturated $I \times I$ design, and define an $(I +1) \times (I+1)$ design containing $\fraction_I$ as:
\begin{equation*}
\fraction_{I+1} = \fraction_I \cup E_{I+1} \ \mbox{ with } \ \fraction_I \cap E_{I+1} = \emptyset \, .
\end{equation*}
Then $\fraction_{I+1}$ is saturated if and only if $E_{I+1}$ has exactly two points, chosen in the union of the $(I+1)$-th row with the $(I+1)$-th column of $\fraction_{I+1}$, with the conditions $m_{A,I+1} \geq 1$ and $m_{B,I+1} \geq 1$.
\end{lemma}

\begin{proof}
$E_{I+1}$ must contain exactly two design points, as $\#\fraction_I=2I-1$ and $\#\fraction_{I+1}=2I+1$. If one or two points are not in the union of the $(I+1)$-th row with the $(I+1)$-th column of $\fraction_{I+1}$, there is a contradiction with Lemma \ref{le:margin1}. If both points are chosen in the $(I+1)$-th row and in the first $I$ columns, the margin $m_{B,I+1}=0$, a contradiction. With an analogous proof, the two points can not be chosen in the $(I+1)$-th column and in the first $I$ rows. All the remaining cases are valid choices, as one among $m_{A,I+1}$ and $m_{B,I+1}$ is equal to $1$ and therefore no cycles can appear.
\end{proof}

We are now ready to approach the problem of computing the number of saturated fractions.

\begin{proposition} \label{pr:np}
Given $m_A=(m_{A,1},\ldots,m_{A,I})$ and $m_B=(m_{B,1},\ldots, m_{B,J})$ with $m_{A,+}=m_{B,+}=I+J-1$, $m_{A,i} \geq 1, i \in [I]$ and $m_{B,j} \geq 1, j \in [J]$, the number of saturated designs with margins equal to $m_A$ and $m_B$ is
\begin{equation} \label{for-multinom}
\left (
\begin{array}[h]{c}
 I-1 \\
 m_{B,1}-1, \ldots, m_{B,J}-1
\end{array}
\right )
\left (
\begin{array}[h]{c}
 J-1 \\
 m_{A,1}-1, \ldots, m_{A,I}-1
\end{array}
\right ) \, .
\end{equation}
\end{proposition}

\begin{proof}
First we consider $J=I$. Without loss of generality, we can assume that the margins of the fraction are arranged in the form:
\begin{eqnarray*}
m_{A,1} \geq m_{A,2} \geq \cdots \geq m_{A,I} = 1 \, , \\
m_{B,1} \geq m_{B,2} \geq \cdots \geq m_{B,I} = 1 \, .
\end{eqnarray*}
Since $m_{A,I}=1$, we can choose a point for the last row, but we have to exclude all design points $(I,h)$ with $m_{B,h}=1$, in order to satisfy the condition in Lemma \ref{le:margin1}, item \ref{it:4}. In the same way, we choose a point in the last column.

We repeat the same argument on the $(I-1) \times (I-1)$ design obtained by deletion of the last row and of the last column. It is immediate to see that both margins of such design have a component equal to $1$. Hence, we iterate $(I-2)$ times the procedure above, until we have a degenerate $1 \times 1$ design with $1$ as its unique element.

If we analyze this algorithm backward, we note that at each step we add two points according to the rule in Lemma \ref{le:margin2}, and therefore the constructed fraction is saturated. Thus, the procedure generates
\begin{equation*}
\left (
\begin{array}[h]{c}
 I-1 \\
 m_{B,1}-1, \ldots, m_{B,I}-1
\end{array}
\right )
\left (
\begin{array}[h]{c}
 I-1 \\
 m_{A,1}-1, \ldots, m_{A,I}-1
\end{array}
\right ) \, ,
\end{equation*}
as each row can be chosen until the margin decreases to $1$, and the same holds for columns.

Now, consider $J>I$. The $B$-margin can be arranged in the form:
\begin{equation*}
m_{B,1} \geq m_{B,2} \geq \cdots \geq m_{B,I} = 1 = m_{B,I+1}=1 = \cdots m_{B,J}=1 \, .
\end{equation*}
With this ordering of the margins, the square $I \times I$ table on the left can be analyzed as in the previous case, while for the last $J-I$ columns there is only the constraint given by the $A$-margin, because the $B$-margin is always equal to $1$, and therefore no $k$-cycles can appear. Hence, the formula in  \eqref{for-multinom} is proved.
\end{proof}

\begin{theorem}\label{th:n_SD}
The number of saturated $I \times J$ designs is $I^{(J-1)}J^{(I-1)}$.
\end{theorem}

\begin{proof}
It follows from Proposition \ref{pr:np} and the classical multinomial theorem, by summation of all possible terms (see Equation \ref{for-multinom}) corresponding to all possible margins.
\end{proof}

We notice that the key consequence of the above enumerations is that the proportion of singular design matrices is not negligible, and it becomes as large as $I$ and $J$ increase.

\begin{corollary}
Let us randomly choose $\fraction \subset I \times J$ with $\#\fraction=I+J-1$. The probability that $\fraction$ is a saturated $I \times J$ design is
\begin{equation*}
\frac{I^{(J-1)}J^{(I-1)}}{\binom{IJ}{I+J-1}}
\end{equation*}
and it tends to $0$ as $I$ and $J$ goes to infinity.
\end{corollary}

For instance, let us consider $I=J$. For $I=3$ we obtain a saturated design in $64\%$ of cases, for $I=4$ in $36\%$ of cases, for $I=5$ in $19\%$ of cases, while for $I=6$ in $10\%$ of cases. Hence, the characterization of non-singular designs, as given in Theorem \ref{th:teo}, is useful from an algorithmic point of view, because the random choice of a subset with $I+J-1$ points does is not an efficient procedure.

\begin{example}\label{ex:I=4}

We discuss here the case $I=J=4$ extensively. The number of saturated designs is $4^6=4096$, corresponding to $36\%$ of designs with $7$ points. The possible configurations of margins, up to the permutation of the levels, are: $(4,1,1,1)$, $(3,2,1,1)$ and $(2,2,2,1)$. The  table below shows the number of saturated design with such margins for one factor, see Proposition \ref{pr:np}, the number of multiset permutations of such configurations of margins and the product of them.
\renewcommand{\arraystretch}{1.5}
\begin{center}
\begin{tabular}{|c|c|c|c|c|} \hline
Margin & $\#$ from Prop. \ref{pr:np} & Multiset permutations &Total \\ \hline
(4,1,1,1) & $\binom{3}{3, 0 , 0 , 0}$ &$\binom{4}{1 , 3}$&$4$ \\ \hline
(3,2,1,1) & $\binom{3}{2 , 1 , 0 , 0}$ &$\binom{4}{1, 1 , 2}$&$36$\\ \hline
(2,2,2,1) & $\binom{3}{1 , 1, 1 , 0}$ &$\binom{4}{3 , 1}$&$ 24$\\ \hline
\end{tabular}
\end{center}
\renewcommand{\arraystretch}{1}
The table below shows the number of saturated design with respect to the two margins.
\begin{center}
\begin{tabular}{ |c|c|c|c|}\hline
    Margins     & (4,1,1,1) & (3,2,1,1) & (2,2,2,1) \\  \hline
 (4,1,1,1)   & $16$ & $144$ & $96$ \\ \hline
  (3,2,1,1) & $144$ &  $1296$ & $864$ \\ \hline
  (2,2,2,1)  & $96$ & $864$ & $576$ \\
  \hline
\end{tabular}
\end{center}
Finally, Figure \ref{fig:tutte4} shows the saturated $4 \times 4$ designs identified by contingency tables. For each margin configuration, a representative of each equivalence class of tables is displayed. The equivalence is up to permutation of rows, permutation of columns, and transposition. We notice that in two cases there is more than one equivalence class.

\begin{figure}\label{fig:tutte4}
\begin{center}
\renewcommand{\tabcolsep}{3pt}
\begin{tiny}
\begin{tabular}{|c|ccc|c|} \hline
\begin{tabular}{c}
\multicolumn{1}{c}{ }\\
\begin{tabular}{|c|c|c|c|} \hline
$\bullet$ & $\bullet$ & $\bullet$ & $\bullet$ \\ \hline
$\bullet$ & & & \\ \hline
$\bullet$ & & & \\ \hline
$\bullet$ & & & \\ \hline
\end{tabular} \\
$(4,1,1,1) $ \\
$(4,1,1,1) $\\ \multicolumn{1}{c}{ }\\
\end{tabular}

&

&

\begin{tabular}{c}
\multicolumn{1}{c}{ }\\
\begin{tabular}{|c|c|c|c|} \hline
$\bullet$ & $\bullet$ & $\bullet$ & \\ \hline
$\bullet$ & & & $\bullet$ \\ \hline
$\bullet$ & & & \\ \hline
$\bullet$ & & & \\ \hline
\end{tabular} \\
$(4,1,1,1) $ \\
$(3,2,1,1) $\\ \multicolumn{1}{c}{ }\\
\end{tabular}

&

&

\begin{tabular}{c}
\multicolumn{1}{c}{ }\\
\begin{tabular}{|c|c|c|c|} \hline
$\bullet$ & $\bullet$ & & \\ \hline
$\bullet$ & & $\bullet$ &  \\ \hline
$\bullet$ & & & $\bullet$ \\ \hline
$\bullet$ & & & \\ \hline
\end{tabular} \\
$(4,1,1,1) $ \\
$(2,2,2,1) $\\ \multicolumn{1}{c}{ }\\
\end{tabular} \\ \hline

&

\begin{tabular}{c}
\multicolumn{1}{c}{ }\\
\begin{tabular}{|c|c|c|c|} \hline
$\bullet$ & $\bullet$ & $\bullet$ & \\ \hline
$\bullet$ & & & $\bullet$ \\ \hline
$\bullet$ & & & \\ \hline
& $\bullet$ & & \\ \hline
\end{tabular} \\
$(3,2,1,1) $ \\
$(3,2,2,1) $\\ \multicolumn{1}{c}{ }\\
\end{tabular}

&

\begin{tabular}{c}
\multicolumn{1}{c}{ }\\
\begin{tabular}{|c|c|c|c|} \hline
$\bullet$ & $\bullet$ & $\bullet$ & \\ \hline
& $\bullet$ & & $\bullet$ \\ \hline
$\bullet$ & & & \\ \hline
$\bullet$ & & & \\ \hline
\end{tabular} \\
$(3,2,1,1) $ \\
$(3,2,1,1) $\\ \multicolumn{1}{c}{ }\\
\end{tabular}

&

\begin{tabular}{c}
\multicolumn{1}{c}{ }\\
\begin{tabular}{|c|c|c|c|} \hline
& $\bullet$ & $\bullet$ & $\bullet$ \\ \hline
$\bullet$ & $\bullet$ & & \\ \hline
$\bullet$ & & & \\ \hline
$\bullet$ & & & \\ \hline
\end{tabular} \\
$(3,2,1,1) $ \\
$(3,2,1,1) $\\ \multicolumn{1}{c}{ }\\
\end{tabular}

&

\begin{tabular}{cc}
\begin{tabular}{c}
\multicolumn{1}{c}{ }\\
\begin{tabular}{|c|c|c|c|} \hline
$\bullet$ & $\bullet$ & & \\ \hline
$\bullet$ & & & $\bullet$ \\ \hline
$\bullet$ & & $\bullet$ & \\ \hline
& $\bullet$ & & \\ \hline
\end{tabular} \\
$(3,2,1,1) $ \\
$(2,2,2,1) $\\ \multicolumn{1}{c}{ }\\
\end{tabular} &

\begin{tabular}{c}
\multicolumn{1}{c}{ }\\
\begin{tabular}{|c|c|c|c|} \hline
$\bullet$ & & & $\bullet$ \\ \hline
$\bullet$ & $\bullet$ & & \\ \hline
& $\bullet$ & $\bullet$ & \\ \hline
$\bullet$ & & & \\ \hline
\end{tabular} \\
$(3,2,1,1) $ \\
$(2,2,2,1) $ \\ \multicolumn{1}{c}{ }
\end{tabular}\\
\end{tabular} \\ \hline

&

&

&

&

\begin{tabular}{c}
\multicolumn{1}{c}{ }\\
\begin{tabular}{|c|c|c|c|} \hline
$\bullet$ & $\bullet$ & & \\ \hline
$\bullet$ & & $\bullet$ & \\ \hline
& $\bullet$ & &$\bullet$ \\ \hline
& &$\bullet$ & \\ \hline
\end{tabular} \\
$(2,2,2,1) $ \\
$(2,2,2,1) $ \\ \multicolumn{1}{c}{ }\\
\end{tabular} \\
 \hline

\end{tabular}
\end{tiny}
\end{center}
\caption{Non-equivalent saturated $4 \times 4$ designs classified by the margins.}
\end{figure}
\renewcommand{\tabcolsep}{6pt}
\end{example}

\section{Markov bases for $I \times J$ designs} \label{mar-bas}

Proposition \ref{pr:k_OA} and Theorem \ref{th:teo} lead to an
interesting connection with the theory of Markov bases for this
kind of experimental design. In fact, we have already identified the two-factor design with a binary $I \times J$ contingency table. Under such representation, we define the matrix $N({\mathcal F})$, where $N({\mathcal F})_{i,j}=1$ if $(i,j)$ is a point of the fraction and $N({\mathcal F})_{i,j}=0$ otherwise. Using Algebraic statistics tools, we
are able to generate all fractions with given margins through a
Markov chain algorithm following the theory in
\cite{rapallo|rogantin:07}.

We briefly recall the basic facts about Markov bases. The reader can refer to the book \cite{drtonetal:09} for a complete presentation. A \emph{Markov move} $m$ is a table with integer entries such that $N({\mathcal F})$, $N({\mathcal F})+m$ and $N({\mathcal F})-m$ have the same margins. A \emph{Markov basis} ${\mathcal M}$ is a finite set of Markov moves which makes connected the set of all tables, or designs, with the same margins. Notice that by adding Markov moves to a Markov basis yields again a Markov basis.

If we start from a fraction with matrix $N({\mathcal F}_0)$, the Markov chain is
then built as follows:
\begin{itemize}
\item at each step $i$, randomly choose a Markov move $m$ in ${\mathcal M}$ and
a sign $\varepsilon \in \{ \pm 1 \}$;

\item if $N({\mathcal F}_i)+ \varepsilon m$ is a binary table,
move the chain to $N({\mathcal F}_{i+1}) =
N({\mathcal F}_i)+ \varepsilon m$; otherwise, stay in
$N({\mathcal F}_i)$.
\end{itemize}
The Markov chain described above is a connected chain over all the
designs with fixed margins, and its stationary distribution is the uniform one. By considering the classical Metropolis-Hastings probability ratio, one can define a Markov chain converging to any specified probability distribution, see \cite{diaconis|sturmfels:98}.

The theory in
\cite{rapallo|rogantin:07}, Sections 4 and 5, states
that the relevant Markov basis for our problem can be computed
from the complete bipartite graph of the design. The \emph{complete bipartite graph} has one vertex for each level of $A$,
one vertex for each level of $B$ and one edge connects each
$A$-vertex with each $B$-vertex. A circuit of degree $k$ is a closed path
with $2k$ vertices, and with edges
\begin{equation}  \label{seq-edges}
(i_1,j_1),(j_1,i_2),(i_2,j_2), \ldots,(i_k,j_k), (j_k,i_1) \, ,
\end{equation}
where $i_1, \ldots, i_k$ are distinct $A$-indices and $j_1,
\ldots, j_k$ are distinct $B$-indices. A concise presentation of
the theory of bipartite graphs can be found in
\cite{sturmfels:96}. The complete bipartite graph for the $3 \times 4$ designs is depicted in Figure \ref{graph}.

\begin{figure} \label{graph}
\begin{center}
\begin{picture}(160,140)(0,0)
\put(30,30){\circle*{4}} \put(3,25){$i=1$}

\put(30,60){\circle*{4}} \put(20,55){$2$}

\put(30,90){\circle*{4}} \put(20,85){$3$}

\put(120,15){\circle*{4}} \put(125,10){$j=1$}

\put(120,45){\circle*{4}} \put(125,40){$2$}

\put(120,75){\circle*{4}} \put(125,70){$3$}

\put(120,105){\circle*{4}} \put(125,100){$4$}

\put(30,30){\line(6,-1){90}} \put(30,30){\line(6,1){90}}
\put(30,30){\line(2,1){90}}  \put(30,30){\line(6,5){90}}

\put(30,60){\line(2,-1){90}} \put(30,60){\line(6,-1){90}}
\put(30,60){\line(6,1){90}} \put(30,60){\line(2,1){90}}

\put(30,90){\line(6,-5){90}} \put(30,90){\line(2,-1){90}}
\put(30,90){\line(6,-1){90}} \put(30,90){\line(6,1){90}}

\end{picture}
\caption{The complete bipartite graph for a $3 \times 4$ design.} \label{bip-circ}
\end{center}
\end{figure}
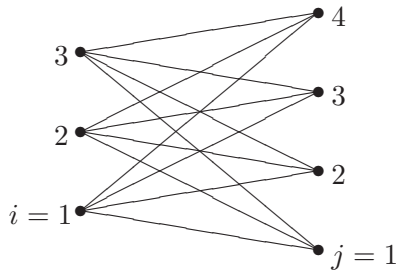

Then, the Markov basis for our problem is defined by associating a Markov move to each circuit of the complete
bipartite graph. Such move has entry $1$ for each edge in even position in the sequence \eqref{seq-edges}, and has entry $-1$ for each edge in odd position.

\begin{proposition}
The $k$-cycles, decomposed into two orthogonal arrays as in Proposition \ref{pr:k_OA}, form a Markov basis.
\end{proposition}
\begin{proof}
It is enough to observe that each circuit of degree $k$ naturally defines a $k$-cycle decomposed as in Proposition \ref{pr:k_OA}.
\end{proof}

\begin{example}
We now discuss the $3 \times 4$ case extensively. The complete bipartite graph in Figure \ref{bip-circ} has $12$
circuits with $4$ edges and $12$ circuits with $6$ edges. Thus, the Markov basis is formed by:
\begin{itemize}
\item $18$ moves of the form
\begin{equation} \label{move:1}
m=\left( \begin{matrix}
1 & -1 & 0 & 0 \\
-1 & 1 & 0 & 0 \\
0 & 0 & 0 & 0 \end{matrix}
   \right)
\end{equation}
corresponding to the $2$-cycles;
\item $24$ moves of the form
\begin{equation} \label{move:2}
m=\left( \begin{matrix}
1 & -1 & 0 & 0 \\
-1 & 0 & 1 & 0 \\
0 & 1 & -1 & 0  \end{matrix}
   \right)
\end{equation}
corresponding to the $3$-cycles.
\end{itemize}
The moves in Equations \eqref{move:1} and \eqref{move:2} are derived from the circuits $(1,1),(1,2),(2,2),(2,1)$ and $(1,1),(1,2),(2,3),(3,3),(3,2),(2,1)$, respectively.

For instance, starting from the fraction with table
\begin{equation*}
N({\mathcal F})=  \left( \begin{matrix}
1 & 1 & 1 & 1 \\
1 & 0 & 0 & 0 \\
1 & 0 & 0 & 0  \end{matrix}
   \right)
   \end{equation*}
we see that no move produces a valid binary table. Thus, we conclude
that $N({\mathcal F)}$ itself is the unique fraction with such margins. On the other hand, starting from the fraction with table
\begin{equation*}
N({\mathcal F})=  \left( \begin{matrix}
1 & 1 & 1 & 0 \\
1 & 0 & 0 & 0 \\
1 & 0 & 0 & 1  \end{matrix}
   \right)
   \end{equation*}
with the Markov moves, we can reach the other two tables
\begin{equation*}
 \left( \begin{matrix}
1 & 0 & 1 & 1 \\
1 & 0 & 0 & 0 \\
1 & 1 & 0 & 0  \end{matrix}
   \right) \ \qquad \
    \left( \begin{matrix}
1 & 1 & 0 & 1 \\
1 & 0 & 0 & 0 \\
1 & 0 & 1 & 0  \end{matrix}
   \right) \, .
\end{equation*}
Notice that a Markov basis preserve the margins, but it generates all the fractions, no matter if they are saturated or not. For example, the fraction
\begin{equation*}
N({\mathcal F})=  \left( \begin{matrix}
1 & 1 & 0 & 1 \\
1 & 0 & 1 & 0 \\
0 & 1 & 0 & 0  \end{matrix} \right)
\end{equation*}
represents a saturated design, while
\begin{equation*}
N({\mathcal F}) +  \left( \begin{matrix}
0 & 0 & 0 & 0 \\
0 & 1 & -1 & 0 \\
0 & -1 & 1 & 0  \end{matrix} \right)
\end{equation*}
does not.
\end{example}

The number of $k$-cycles increases fast with $I$ and $J$, and the computations become unfeasible for large $I$ or $J$. However, with some additional hypotheses on the fractions, for example assuming strict positivity of all margins, the connectedness of the Markov chain is ensured by a subset of circuits. That computational issues are studied in \cite{hara|takemura:10} and \cite{rapallo|yoshida:10}.

The results above yield another method for enumerating the saturated fractions. In fact, from Theorem \ref{th:teo}, a saturated fraction must have size equal to $I+J-1$ and it must not contain cycles. For those readers familiar with the language of graph theory, this implies that, in terms of the complete bipartite graph, the subgraph corresponding to the fraction must be a spanning tree. A known result in graph theory, named as the Cayley's formula, says that the number of spanning trees in a complete bipartite graph is $I^{J-1}J^{I-1}$. This represents an alternative proof of theorem \ref{th:n_SD}. For details on complete bipartite graphs and the Cayley's formula see \cite{rosen|michaels:00}.

\section{Conclusions} \label{fut-dir}

The theory described in this paper suggests several extensions and applications. First, it is interesting to explore how the results can be extended for the characterization of saturated fractions to more general designs. In the multi-way setting the generalization of $k$-cycles is not trivial and more advanced notions of Algebraic statistics must be used. Some preliminary examples are being explored by the authors and the findings seem fruitful. Second, the connections between fractions and graphs, as suggested in the last paragraph of Section \ref{mar-bas}, can lead to new useful results. Furthermore, it would be interesting to study the classification of the saturated fractions with respect to some statistical criteria. Among these criteria, we cite the minimum aberration in a classical sense, or more recent tools, such as state polytopes. Minimum aberration is a classical notion in this framework, and is supported by a large amount of literature. This theory has been developed in \cite{fries|hunter:80} and more recently in, e.g., \cite{ye:03} with the use of the indicator function in the two-level case. The extension to the multilevel case is currently an open problem. The use of state polytopes has been introduced in \cite{bersteinetal:10}. From the point of view of applications, the use of the inequivalent saturated fractions to perform exact tests on model parameters is worth studying, together with its implementation in statistical softwares, such as SAS or R. For inequivalent orthogonal arrays very interesting results have already been achieved, see \cite{basso} and \cite{arbfonrag}.

\bibliographystyle{spmpsci}
\bibliography{refsfrr}

{\bf Acknowledgment.} The authors would like to thank Professor Bernd Sturmfels (U.C. Berkeley) for his useful suggestions. RF acknowledges SAS institute for providing software. FR is partially supported by the PRIN2009 grant number 2009H8WPX5.

\end{document}